\documentclass[12pt]{article}

\usepackage[T1]{fontenc}
\usepackage[utf8]{inputenc}
\usepackage{authblk}
\usepackage{color}

\usepackage{amsmath, amssymb, amsfonts, amsbsy, amsthm, latexsym, color}
\usepackage{graphicx}
\usepackage{enumerate}
\everymath{\displaystyle}
\newtheorem{theorem}{Theorem}[section]
\newtheorem{proposition}[theorem]{Proposition}

\newtheorem{corollary}[theorem]{Corollary}

\numberwithin{equation}{section}

\begin{document}

\title{Optimal Parseval frames: Total coherence and total volume}

\author[1]{Jameson Cahill}
\author[2]{Peter G. Casazza}
\affil[1]{Department of Mathematical Sciences, New Mexico State University} 
\affil[2]{Department of Mathematics, University of Missouri}

\date{}

\maketitle

\begin{abstract}
We introduce three quantities called total coherence, total volume, and nuclear energy, and we show that equiangular Parseval frames maximize all three of these quantities over the set of all Parseval frames. We then show that equiangular Parseval frames also maximize the total volume and nuclear energy over the set of equal norm frames. Along the way we derive a bound on the smallest k-dimensional volume in an equal norm frame which is a generalization of the Welch bound. We conclude with an extensive list of open questions.
\end{abstract}

\section{Introduction}

A \textit{finite frame} for an $N$-dimensional vector space $\mathbb{F}^N$ (where $\mathbb{F}$ can be either $\mathbb{R}$ or $\mathbb{C}$) is simply a finite collection of vectors $\Phi=\{\varphi_i\}_{i=1}^M$ that spans $\mathbb{F}^N$. By a slight abuse of notation we will use the same symbol $\Phi$ to refer to the $N\times M$ matrix whose $i$th column is $\varphi_i$. Two important operators associated with a given frame are the \textit{frame operator} $\Phi\Phi^*=\sum_{i\in[M]}\varphi_i\varphi_i^*$ and the \textit{Gram matrix} $\Phi^*\Phi$. Given a subset $K\subseteq [M]$ we denote by $\Phi_K$ the $N\times |K|$ matrix consisting of the columns of $\Phi$ indexed by $K$ and we define the \textit{partial frame operator} $\Phi_K\Phi_K^*=\sum_{i\in K}\varphi_i\varphi_i^*$.

A frame $\Phi$ is called a \textit{Parseval frame} if the frame operator is the identity operator, i.e., $\Phi\Phi^*=I$. It is easy to see that this is equivalent to the Gram matrix $\Phi^*\Phi$ being an orthogonal projection onto an $N$-dimensional subspace of $\mathbb{F}^M$. Two Parseval frames $\{\varphi_i\}_{i=1}^M$ and $\{\psi_i\}_{i=1}^M$ are called \textit{unitarily equivalent} if there is a unitary operator $U$ so that $\varphi_i=U\psi_i$ for every $i\in[M]$. It follows that two Parseval frames have the same Gram matrix if and only if they are unitarily equivalent. Furthermore, every $M\times M$ orthogonal projection of rank $N$ is the Gram matrix some Parseval frame for $\mathbb{F}^N$, so there is a one to one correspondence between unitary equivalence classes of Parseval frames for $\mathbb{F}^N$ and the Grassmannian of $N$-dimensional subspaces of $\mathbb{F}^M$ which we will denote by $Gr(M,N)$.

A frame $\Phi=\{\varphi_i\}_{i=1}^M$ is called \textit{equal norm} if $\|\varphi_i\|=\|\varphi_j\|$ for every $i$ and $j$. If $\Phi$ is also a Parseval frame then we know $\text{trace}(\Phi^*\Phi)=N$, and since the diagonal entries of $\Phi^*\Phi$ are $\|\varphi_i\|^2$ we see that $\|\varphi_i\|=\sqrt{N/M}$. An equal norm Parseval frame which further satisfies $|\langle \varphi_i,\varphi_j\rangle|=c$ for every $i\neq j$ and some constant $c$ is called \textit{equiangular}. In this case we have
\begin{eqnarray*}
N = \|\Phi^*\Phi\|_2^2 &=& \text{trace}(\Phi^*\Phi\Phi^*\Phi) \\
&=& \text{trace}(\Phi^*\Phi)=\sum_{i,j}|\langle\varphi_i,\varphi_j\rangle|^2 \\
&=& \sum_{i=1}^M\|\varphi_i\|^4+\sum_{i\neq j}|\langle\varphi_i,\varphi_j\rangle|^2 \\
&=&\frac{N^2}{M}+M(M-1)c^2,
\end{eqnarray*}
so we have
\begin{equation}\label{welch}
c_{M,N}:=c=\sqrt{\frac{N(M-N)}{M^2(M-1)}}.
\end{equation}

As a precaution to the reader we remark here that it is more common in the literature to study equiangular tight frames, which are the same object as an equiangular Parseval frame except the vectors are rescaled to be unit vectors. As we will see in this paper, equiangular Parseval frames tend to be optimal for a variety of applications. However they are not guaranteed to exist, and in fact for most pairs $(M,N)$ they do not exist. See \cite{FM15} for a list of known examples of equiangular Parseval frames.

If $\Phi=\{\varphi_i\}_{i=1}^M$ is a Parseval frame for $\mathbb{F}^N$, then any frame $\Psi=\{\psi_i\}_{i=1}^M$ for $\mathbb{F}^{M-N}$ which satisfies $\Psi^*\Psi=I-\Phi^*\Phi$ is called a \textit{Naimark complement} of $\Phi$ (note that Naimark complements are only defined up to unitary equivalence). It is clear from the definition that the properties of being equal norm or equiangular are preserved under Naimark complementation, since $\|\psi_i\|^2=1-\|\varphi_i\|^2$ and $\langle\psi_i,\psi_j\rangle=-\langle\varphi_i,\varphi_j\rangle$.

For more background on finite frames we refer to the book \cite{CK12}, particularly the first chapter, and the more recent book \cite{W18}.

This paper is organized as follows. In section \ref{tc} we introduce the concept of total coherence and pose a related optimization problem and record some results about the solutions to this problem. In section \ref{tv} we introduce the notion of total $k$-dimensional volume and in section \ref{ne} we introduce the $k$-nuclear energy as well as optimization problems related to these concepts over the set of Parseval frames. In section \ref{ens} we consider these same problems over the set of equal norm frames. In section \ref{dc} we pose an extensive list of questions. Some of these questions may be quite straightforward to answer and could be approachable by students while others are quite difficult and may be nearly impossible to answer completely.

\section{Total coherence}\label{tc}

Let $\mathcal{P}(M,N)$ denote the space of Parseval frames for $\mathbb{F}^N$ with $M$ vectors.  We would like to study the frames that solve the following optimization problem:
\begin{equation}\label{1}
\max_{\Phi\in\mathcal{P}(M,N)}\sum_{i,j}|\langle \varphi_i,\varphi_j\rangle|.
\end{equation}

Note that
\begin{eqnarray*}
\sum_{i,j}|\langle \varphi_i,\varphi_j\rangle|&=&\sum_{i=1}^M\|\varphi_i\|^2+\sum_{i\neq j}|\langle \varphi_i,\varphi_j\rangle| \\
&=& N+\sum_{i\neq j}|\langle \varphi_i,\varphi_j\rangle|
\end{eqnarray*}
so \eqref{1} is equivalent to
\begin{equation}\label{2}
\max_{\Phi\in\mathcal{P}(M,N)}\sum_{i\neq j}|\langle \varphi_i,\varphi_j\rangle|.
\end{equation}

For a Parseval frame $\Phi$ we call the quantity $TC(\Phi)=\sum_{i\neq j}|\langle\varphi_i,\varphi_j\rangle|$ the \textit{total coherence} of $\Phi$. One observation that we will make right away is that total coherence is preserved under unitary equivalence and Naimark complementation. We will state this as a proposition for later reference.

\begin{proposition}\label{tcnaimark}
If two Parseval frames $\Phi$ and $\Psi$ are Naimark complements then $TC(\Phi)=TC(\Psi)$.
\end{proposition}

An immediate consequence of the above proposition is that if $\Phi$ solves \eqref{2} for $\mathcal{P}(M,N)$ then any Naimark complement of $\Phi$ solves \eqref{2} for $\mathcal{P}(M,M-N)$. This means that for understanding the solutions to \eqref{2} we can always assume either $M\geq 2N$ or $N<M\leq 2N$ depending on which is more convenient.

\begin{theorem}\label{main}
If $\Phi=\{\varphi_i\}_{i=1}^M$ is an equiangular Parseval frame and $\Psi=\{\psi_i\}_{i=1}^M\in\mathcal{P}(M,N)$ then
$$
TC(\Psi)\leq TC(\Phi).
$$
Thus, when they exist, equiangular Parseval frames are precisely the solutions to \ref{2} (and equivalently \ref{1}).
\end{theorem}
\begin{proof}
First note that if $\Phi$ is an equiangular Parseval frame then by \eqref{welch}
$$
|\langle \varphi_i,\varphi_j\rangle|=\sqrt{\frac{N(M-N)}{M^2(M-1)}}
$$
whenever $i\neq j$, so
\begin{eqnarray*}
\sum_{i\neq j}|\langle\varphi_i,\varphi_j\rangle|&=& M(M-1)\sqrt{\frac{N(M-N)}{M^2(M-1)}} \\
&=&\sqrt{N(M-N)(M-1)}.
\end{eqnarray*}

Let $\{\psi_i\}_{i=1}^M\in\mathcal{P}(M,N)$. Then
\begin{eqnarray*}
N&=&\sum_{i,j}|\langle\psi_i,\psi_j\rangle|^2 \\
&=&\sum_{i=1}^M\|\psi_i\|^4+\sum_{i\neq j}|\langle\psi_i,\psi_j\rangle|^2.
\end{eqnarray*}

But we also have that $N=\sum_{i=1}^M\|\psi_i\|^2$ which means that $\sum_{i=1}^M\|\psi_i\|^4\geq \frac{N^2}{M}$ with equality if and only if $\Psi$ is equal norm, and so
\begin{eqnarray*}
\sum_{i\neq j}|\langle\psi_i,\psi_j \rangle|^2&\leq&N-\frac{N^2}{M} \\
&=& \frac{N(M-N)}{M}.
\end{eqnarray*}

Therefore
\begin{eqnarray*}
\sum_{i\neq j}|\langle \psi_i,\psi_j\rangle|&\leq&(M(M-1)\sum_{i\neq j}|\langle\psi_i,\psi_j\rangle|^2)^{\frac{1}{2}} \\
&\leq&(M(M-1)(\frac{N(M-N)}{M}))^{\frac{1}{2}} \\
&=&\sqrt{N(M-N)(M-1)} \\
&=&\sum_{i\neq j}|\langle\varphi_i,\varphi_j\rangle|.
\end{eqnarray*}
\end{proof}

Note that the key ingredient in the above proof is the following property of equal norm Parseval frames: Although the $\ell^2$-energy of the Gram matrix of a Parseval frame is constant for all Parseval frames in $\mathcal{P}(M,N)$, this energy is not equally distributed between the diagonal and the off-diagonal entries of the Gram matrix. Equal norm Parseval frames minimize the amount of this energy that is on the diagonal (i.e., the norms of the vectors) and therefore maximizes the amount that is in the off-diagonal entries (i.e., inner products between different vectors). This proof should be compared to the derivation of the Welch bound \cite{W74} (see also \cite{DHC12}).

\begin{proposition}\label{tcbound}
If $\Psi$ is a Parseval frame that solves \eqref{2} then
$$
\max\{N,M-N\}\leq TC(\Psi)\leq\sqrt{N(M-N)(M-1)}.
$$
\end{proposition}

\begin{proof}
The upper bound follows from Theorem \ref{main}. For the lower bound let $\Phi=\{\varphi_i\}_{i=1}^M$ be any equal norm Parseval frame. Then for a fixed $i$ we have
\begin{align*}
\frac{N}{M}&= \sum_{j=1}^M|\langle \varphi_i,\varphi_j\rangle|^2\\
&\le \sum_{j=1}^M\|\varphi_i\|\|\varphi_j\||\langle \varphi_i,
\varphi_j\rangle|\\
&= \frac{N}{M}\sum_{j=1}^M|\langle \varphi_i,\varphi_j\rangle|,
\end{align*}
so $\sum_{j=1}^M|\langle\varphi_i,\varphi_j\rangle|\geq 1$. Since this is true for every $i$ we have $\sum_{i,j}|\langle\varphi_i,\varphi_j\rangle|\geq M$ which means $TC(\Phi)\geq M-N$. 

If $\Gamma$ is a Naimark complement of $\Phi$ then by the same argument $TC(\Gamma)\geq M-(M-N)=N$, but by Proposition \ref{tcnaimark} $TC(\Phi)=TC(\Gamma)$.

Finally, since $\Psi$ is a solution to \eqref{2} we know $TC(\Psi)\geq TC(\Phi)$.
\end{proof}

Note that to prove the lower bound in the above proposition we just used an arbitrary equal norm Parseval frame. Given that we know that when they exist equiangular Parseval frames are the solutions to \eqref{2} it is natural to ask whether the solutions to \eqref{2} are always equal norm. While we do not have a proof of this (and it is quite possible that it is not true) we can still show that there has to be some control on the norms of the vectors in a Parseval frame that solves \eqref{2}.

\begin{theorem}\label{tcnorm}
For each $M>N$ there exist constants $0<c\leq d<1$ (depending only on $M$ and $N$) such that if $\Phi=\{\varphi_i\}_{i=1}^M$ solves \eqref{2} then $c\leq\|\varphi_i\|\leq d$ for every $i=1,...,M$.
\end{theorem}

\begin{proof}
Assume $\Psi=\{\psi_i\}_{i=1}^M$ is a Parseval frame with $\psi_M=0$ and $\psi_{M-1}\neq 0$.  Define $\Gamma=\{\gamma_i\}_{i=1}^M$ by:
\[ \gamma_i = \begin{cases}
\psi_i \mbox{ if } i=1,2,\ldots,M-2\\
\frac{\psi_{M-1}}{\sqrt{2}}\mbox{ if } i=M-1,M.
\end{cases}\]
It is easy to see that $\Gamma$ is a Parseval frame since $\sum\gamma_i\gamma_i^*=\sum\psi_i\psi_i^*$. We have
\begin{align*}
TC(\Gamma)&=\sum_{\substack{i,j=1 \\ i\neq j}}^{M-2}|\langle\gamma_i,\gamma_j\rangle|+2|\langle\gamma_{M-1},\gamma_M\rangle|+2\sum_{i=1}^{M-2}(|\langle\gamma_{M-1},\gamma_i\rangle|+|\langle\gamma_M,\gamma_i\rangle|) \\
&=\sum_{\substack{i,j=1 \\ i\neq j}}^{M-2}|\langle\psi_i,\psi_j\rangle|+2|\langle\frac{\psi_{M-1}}{\sqrt{2}},\frac{\psi_{M-1}}{\sqrt{2}}\rangle|+4\sum_{i=1}^{M-2}|\langle\frac{\psi_{M-1}}{\sqrt{2}},\psi_i\rangle| \\
&=\sum_{\substack{i,j=1 \\ i\neq j}}^{M-2}|\langle\psi_i,\psi_j\rangle|+\|\psi_{M-1}\|^2+2\sqrt{2}\sum_{i=1}^{M-2}|\langle\psi_{M-1},\psi_i\rangle| \\
&>\sum_{\substack{i,j=1 \\ i\neq j}}^{M-2}|\langle\psi_i,\psi_j\rangle|+2\sum_{i=1}^{M-2}|\langle\psi_{M-1},\psi_i\rangle|=TC(\Psi).
\end{align*}
This shows that $\|\varphi_i\|>0$. To see that $\|\varphi_i\|<1$ observe that a Parseval frame contains a unit vector if and only if the corresponding vector in any Naimark complement is the zero vector, and then apply Proposition \ref{tcnaimark}.

Let $TC(M,N)$ denote the set of Gram matrices of Parsevals frames that solve \eqref{2}. $TC(M,N)$ is a closed subset of $Gr(M,N)$ since it is the preimage of the optimal value of the continuous function $TC(\Phi)$, and since $Gr(M,N)$ is compact it follows that $TC(M,N)$ is also compact. Now consider the value $\|\varphi_1\|^2$, which is just the top left entry of the Gram matrix and is therefore a continuous function on $Gr(M,N)$ so it achieves a maximum and a minimum on $TC(M,N)$. Also note that $\{\varphi_i\}_{i=1}^M\in TC(M,N)$ if and only if $\{\varphi_{\sigma(i)}\}_{i=1}^M$ for any permutation $\sigma$ of the index set $[M]$. Let
$$
c^2=\min_{\Phi\in TC(M,N)}\|\varphi_1\|^2
$$
and
$$
d^2=\max_{\Phi\in TC(M,N)}\|\varphi_1\|^2.
$$
\end{proof}

Given a Parseval frame $\Phi$ consider the following quantity:
\begin{equation*}\label{E1}
 EAD(\Phi)=\sum_{i=1}^M (\|\varphi_i\|^2- \frac{N}{M})^2+
\sum_{i\not= j}(|\langle \varphi_i,\varphi_j\rangle|-c_{M,N})^2
\end{equation*}
where $c_{M,N}$ is the constant derived in \eqref{welch}. We call this quantity the \textit{equiangular distance} of $\Phi$ since it is the Froebenius distance between the Gram matrix of $\Phi$ and the Gram matrix of a (possibly nonexistent) equiangular Parseval frame. Note that $\Phi$ is an equiangular Parseval frame if and only if $EAD(\Phi)=0$, but we know that in many cases there are no equiangular Parseval frames, so by minimizing this quantity we should find the Parseval frame that is "closest" to being equiangular. So we pose the following minimization problem:
\begin{equation}\label{EAD}
\min_{\Phi\in \mathcal{P}(M,N)}EAD(\Phi).
\end{equation}

Since $\Phi$ is Parseval we know $\sum\|\varphi_i\|^2=\text{trace}\Phi^*\Phi=N$ so the first sum in $EAD(\Phi)$ is precisely the variance in the square norms of the frame vectors and this term is equal to $0$ if and only if $\Phi$ is equal norm. However, the second sum does not represent the variance in the magnitudes of the off-diagonal entries of the Gram matrix unless $\Phi$ is equiangular. Therefore we also pose the following problem:
\begin{equation}\label{var}
\min_{\Phi\in \mathcal{P}(M,N)}V(\Phi):=\sum_{i=1}^M (\|\varphi_i\|^2- \frac{N}{M})^2+
\sum_{i\not= j}(|\langle \varphi_i,\varphi_j\rangle|-c_{\Phi})^2
\end{equation}
where
$$
c_{\Phi}=\frac{TC(\Phi)}{M(M-1)}
$$
which is the mean of the off-diagonal entries. The solutions to this problem minimize the sum of the variance of the diagonal entries and the variance of the magnitudes of the off-diagonal entries of the Gram matrix.

\begin{proposition}
\eqref{2},\eqref{EAD}, and \eqref{var} all have the same solutions.
\end{proposition}

\begin{proof}

We compute:
\begin{align*}
EAD(\Phi)&= \sum_{i=1}^M\|\varphi_i\|^4 -\frac{2N}{M}\sum_{i=1}^M\|\varphi_i\|^2 + \frac{N^2}{M} + \sum_{i\not= j}|\langle \varphi_i,\varphi_j\rangle|^2
-2c_{M,N} \sum_{i\not= j}|\langle \varphi_i,\varphi_j\rangle|
+ M(M-1)c_{M,N}^2\\
&= \sum_{i=1}^M\|\varphi_i\|^4 + \sum_{i\not= j}|\langle \varphi_i,\varphi_j\rangle|^2 - \frac{N^2}{M}+M(M-1)c_{M,N}^2-2c_{M,N}TC(\Phi) \\
&= N - \frac{N^2}{M}+M(M-1)c_{M,N}^2-2c_{M,N}TC(\Phi).
\end{align*}
Since $c_{M,N}$ only depends on $M$ and $N$ but not on $\Phi$ we see that $EAD(\Phi)$ is minimized precisely when $TC(\Phi)$ is maximized. This shows that \eqref{2} and \eqref{EAD} have the same solutions.

By a similar calculation we see that
\begin{align*}
V(\Phi) &=N-\frac{N^2}{M}+M(M-1)c_{\Phi}^2-2c_{\Phi}TC(\Phi) \\
&=\frac{N(M-N)}{M}+\frac{TC(\Phi)^2}{M(M-1)}-\frac{2TC(\Phi)^2}{M(M-1)} \\
&=\frac{N(M-N)}{M}-\frac{TC(\Phi)^2}{M(M-1)}.
\end{align*}
This is also minimized precisely when $TC(\Phi)$ is maximized so this shows that \eqref{2} and \eqref{var} have the same solutions.
\end{proof}

Note that since $V(\Phi)$ must be nonnegative the last line of the above proof shows that $TC(\Phi)\leq\sqrt{N(M-N)(M-1)}$ providing an alternative proof to Theorem \ref{main}.

Finally, we consider the problem
\begin{equation}\label{AD}
\min_{\Phi\in \mathcal{P}(M,N)} AD(\Phi):=\sum_{i\not= j}(|\langle \varphi_i,\varphi_j\rangle|-c_{M,N})^2
\end{equation}
where $c_{M,N}$ is given in \eqref{welch}.

\begin{proposition}
If $\Phi$ is a solution to \eqref{AD} which is equal norm then $\Phi$ is also a solution to \eqref{EAD} (and equivalently \eqref{2} and \eqref{var}) and all solutions to \eqref{EAD} (and equivalently \eqref{2} and \eqref{var}) are equal norm.
\end{proposition}
\begin{proof}
First note that for any Parseval frame $\Phi$, $EAD(\Phi)=AD(\Phi)$ if and only if $\Phi$ is equal norm. Let $\Phi$ be an equal norm solution to \eqref{AD} and let $\Psi$ be any solution to \eqref{EAD}. We have
$$
AD(\Phi)\leq AD(\Psi)
$$
by assumption. We also have that
$$
AD(\Psi)\leq EAD(\Psi)\leq EAD(\Phi)=AD(\Phi),
$$
where the first inequality follows from the definitions of $AD$ and $EAD$, the second from our assumption that $\Psi$ is a solution to \eqref{EAD}, and the last equality since $\Phi$ is equal norm. Together these inequalities imply that these quantities are all equal, so $EAD(\Phi)=EAD(\Psi)$ means that $\Phi$ is also a solution to \eqref{EAD}, and $AD(\Psi)=EAD(\Psi)$ means that $\Psi$ is equal norm.
\end{proof}

We remark here that in order to prove the above result we need to consider the quantity $AD(\Phi)$ which is not the variance of the magnitudes of the off-diagonal entries of the Gram matrix. To illustrate this consider the Parseval frame $\Phi=\{\varphi_i\}_{i=1}^M$ where $\{\varphi_i\}_{i=1}^N$ is an orthonormal basis for $\mathbb{F}^N$ and $\varphi_{N+1}=\cdots =\varphi_M=0$. For such a frame we have $TC(\Phi)=0$ and therefore $V(\Phi)=N(M-N)/M$ which is the biggest that $V$ can possibly be. However, when considering the two sums in the definition of $V(\Phi)$ we see that this frame gets all of its variance from the first term which corresponds to the variance in the norms of the vectors. Therefore if we tried to minimize the second term in $V$ then this frame would be a solution to that problem since $\langle\varphi_i,\varphi_j\rangle=0$ for every $i\neq j$.

\section{Total volume}\label{tv}

Given a collection of vectors $\{f_i\}_{i=1}^k\subseteq\mathbb{F}^N$ with $k\leq N$ the \textit{k-dimensional volume} of the parallelotope determined by these vectors is
$$
v_k(\{f_i\}_{i=1}^k)=v_k(F)=\sqrt{\det(F^*F)}=\prod_{i=1}^k\sigma_i(F)
$$
where $F$ is the $N\times k$ matrix with columns $\{f_i\}_{i=1}^k$ and $\sigma_i(F)$ denotes the singular values of $F$ in decreasing order. This volume is $0$ if and only if the vectors are linearly dependent in which case $\sigma_k(F)=0$ and they do not determine a k-dimensional parallelotope.

Given a Parseval frame $\Phi$ for $\mathbb{F}^N$ and $k\leq N$ we define the \textit{total k-dimensional volume} of $\Phi$ by
$$
V_k(\Phi)=\sum_{|K|=k}v_k(\Phi_K),
$$
and we pose the following optimization problem:
\begin{equation}\label{volume}
\max_{\Phi\in \mathcal{P}(M,N)}V_k(\Phi).
\end{equation}

\begin{proposition}\label{prodcos}
If $\Phi\in\mathcal{P}(M,N)$ then
$$
\sum_{|K|=k}v_k^2(\Phi_K)= {N\choose k}.
$$
\end{proposition}
\begin{proof}
Note that $\sum_{|K|=k}v_k^2(\Phi_K)$ is the sum of all principal $k\times k$ minors of $\Phi^*\Phi$ so $(-1)^k\sum_{|K|=k}v_k^2(\Phi_K)$ is the coefficient of $\lambda^{M-k}$ in the characteristic polynomial of $\Phi^*\Phi$ (see section 7.1 of \cite{LAbook}). But since $\Phi$ is a Parseval frame the characteristic polynomial of $\Phi^*\Phi$ is $\lambda^{M-N}(\lambda-1)^N$ so the coefficient of $\lambda^{M-k}$ is $(-1)^k{N\choose k}$.

See Theorem 1 in \cite{MB96} for a different proof of a more geometric nature.
\end{proof}
From this we see that if all of the k-dimensional parallelotopes determined by $\Phi$ have the same volume $c_{M,N,k}$ then
$$
c_{M,N,k}=\sqrt{{M\choose k}^{-1}{N\choose k}}=\sqrt{\frac{N!(M-k)!}{M!(N-k)!}}.
$$
The case $k=1$ tells us the familiar fact that if $\Phi$ is equal norm then $\|\varphi_i\|=\sqrt{N/M}$ for every $i$. 

\begin{proposition}\label{volbound}
Let $\Phi\in\mathcal{P}(M,N)$ and $k\leq N$, then
$$
{N\choose k}\leq V_k(\Phi)\leq \sqrt{{M\choose k}{N \choose k}}
$$
with equality in the upper bound if and only if $v_k(\Phi_K)=c_{M,N,k}$ for every $K$ with $|K|=k$, and equality in the lower bound if and only if there is a subset $J\subseteq [M]$ with $|J|=N$ such that $\{\varphi_i\}_{i\in I}$ is an orthonormal basis and $\varphi_i=0$ whenever $i\not\in J$.
\end{proposition}
\begin{proof}
The upper bound follows directly from Proposition \ref{prodcos}. For the lower bound we have
\begin{eqnarray*}
{N\choose k}&=&\sum_{|K|=k}v_k^2(\Phi_K) \\
&\leq& \max_{|K|=k}\{v_k(\Phi_K)\}V_k(\Phi)\leq V_k(\Phi)
\end{eqnarray*}
were the second inequality follows from the fact that $\|\varphi_i\|\leq 1$ for every $i$ since $\Phi$ is Parseval. To achieve equality in the first inequality we must have $v_k(\Phi_K)=\max\{v_k(\Phi_K)\}$ whenever $v_k(\Phi_K)>0$, and to achieve equality in the second inequality we must have $\max\{v_k(\Phi_K)\}=1$. Since $\|\varphi_i\|\leq 1$ it follows that $v_k(\Phi_K)=1$ if and only if $\Phi_K$ is orthonormal. It now follows that a Parseval frame consisting of an orthonormal basis and zeros will satisfy both inequalities with equality. 

If $\Phi$ is any other Parseval frame then it contains a vector $\varphi_i$ with $0<\|\varphi_i\|<1$. If $K$ is any set with $|K|=k$, $i\in K$ and $\Phi_K$ linearly independent (which must exist since $\Phi$ is a frame) then $0<v_k(\Phi_K)< 1$. If $\Phi$ has the property that $v_k(\Phi_K)=\max\{v_k(\Phi_K)\}$ whenever $v_k(\Phi_K)>0$ then $\max\{v_k(\Phi_K)\}<1$ and the second inequality is strict, otherwise the first inequality is strict.
\end{proof}

For a given Parseval frame $\Phi$ the variance in the k-dimensional volumes determined by $\Phi$ is
\begin{eqnarray}
Var_k(\Phi)&=&\sum_{|K|=k}(v_k(\Phi_K)-{M\choose k}^{-1}V_k(\Phi))^2 \label{volvar}\\
&=&{N\choose k}-{M\choose k}^{-1}V_k^2(\Phi), \notag
\end{eqnarray}
so $Var_k(\Phi)$ is minimized precisely when $V_k(\Phi)$ is maximized. From this perspective it is easy to see that if $v_k(\Phi)=c_{M,N,k}$ for every $K$ then $\Phi$ is a solution to \eqref{volume}. This establishes that the solutions to \eqref{volume} when $k=1$ are precisely the equal norm Parseval frames. It also shows that when equiangular Parseval frames exist then they are solutions to \eqref{volume} for $k=2$, however it does not show that there are no other solutions in this case.

It is known that if a Parseval frame $\Phi=\{\varphi_i\}_{i=1}^M$ satisfies
\begin{equation}\label{angle}
|\langle \frac{\varphi_i}{\|\varphi_i\|},\frac{\varphi_j}{\|\varphi_j\|}\rangle|=c
\end{equation}
for every $i\neq j$ then in fact $\Phi$ must be equal norm (see \cite{BPT09}). This means that if we define the angle between vectors as in \eqref{angle} (which is the correct way) then equiangular Parseval frames are automatically equal norm, so we do not need to include this in the definition. We now prove a similar result for the case of $k$-dimensional volumes.

\begin{theorem}\label{equalvol}
Let $\{\varphi_i\}_{i=1}^M$ be a Parseval frame for $\mathbb{F}^N$ and assume there
is a $0<k\leq N$ so that the parallelotopes spanned by any $k$-elment subset of the
frame have the same volume.  Then the $k-1$-dimensional parallelotopes also have the same volume.
\end{theorem}

\begin{proof}
Choose any $k-1$ element subset of the Parseval frame say $\{\varphi_i\}_{i\in J}$.
For each $m,n\notin J$, 
\[ v_k(\{\varphi_i\}_{i\in J} \cup \{\varphi_m\}) = v_k(\{\varphi_i\}_{i\in J}\cup \{\varphi_n\})=c_{M,N,k}.\]
Letting $P$ be the orthogonal projection of the space onto span $\{\varphi_i\}_{i\in J}$,
the above says that
\[ \|(I-P)\varphi_m\|= \|(I-P)\varphi_n\|:= d.\]
Since this is a Parseval frame,
\[ N-k = \sum_{j\notin J}\|(I-P)\varphi_j\|^2 = (M-k)d^2.\]
So
\[ \|(I-P)\varphi_j\|^2= d^2=\frac{N-k}{M-k}.\]
Finally, we have
\[ c_{M,N,k}=v_k(\{\varphi_i\}_{i\in J}\cup \{\varphi_j\}) = d\cdot v_{k-1}(\{\varphi_i\}_{i\in J}).\]
That is,
\[v_{k-1}(\Phi_J)= v_{k-1}(\{\varphi_i\}_{i\in J}) = \frac{c_{M,N,k}}{d}=c_{M,N,k-1}.\]
\end{proof}

\begin{corollary}\label{C1}
Let $\Phi$ be a Parseval frame for $\mathbb{F}^N$ and assume there
is a $2\leq k\leq N$ so that the parallelotopes spanned by any $k$-elment subset of the
frame have the same volume.  Then $\Phi$ is equiangular.
\end{corollary}

\begin{proof}
By iterating this result down to $k=1$, we discover that the 1-dimensional
volumes are all equal which means $\Phi$ is equal norm. For $k=2$ we see that the parallelogram determined by any pair of vectors in $\Phi$ have the same area, but since $\Phi$ is equal norm this implies $\Phi$ is equiangular.
\end{proof}

This also shows the following which we state as a corollary for later reference:

\begin{corollary}\label{eav}
The solutions to \eqref{volume} for $k=2$ are precisely the equiangular Parseval frames when they exist.
\end{corollary}

Since we know that the properties of being equal norm and equiangular are preserved under Naimark complementation it is natural to ask whether this is still true for $k$-dimensional volumes for $k>2$.

\begin{proposition}\label{naimarkvol}
Let $\Phi$ be a Parseval frame and suppose $\Psi$ is a Naimark complement to $\Phi$ and suppose $k\leq\min\{N,M-N\}$. If $v_k(\Phi_K)=c_{M,N,k}$ for every $K$ then $v_k(\Psi_K)=c_{M,M-N,k}$ for every $K$.
\end{proposition}
\begin{proof}
This follows from the identity
\begin{equation}\label{detid}
\det(I+A)=1+\sum \det(A_J)
\end{equation}
where the sum is taken over all principal submatrices of $A$ (including $A$ itself). For $K\subseteq [M]$ we have that $\Psi_K^*\Psi_K=I-\Phi_K^*\Phi_K$, so we can apply \eqref{detid} with $A=-\Phi_K^*\Phi_K$. The determinants of the principal submatrix of $\Phi_K^*\Phi_K$ indexed by $J\subseteq K$ corresponds to the  $|J|$-dimensional volume of the parallelotope spanned by $\{\varphi_i\}_{i\in J}$ which is equal to $c_{M,N,|J|}^2$ by Theorem \ref{equalvol}. 
\end{proof}

There is a subtlety in Proposition \ref{naimarkvol} that is worth addressing. In order to interpret $v_k(\Phi_K)$ as a $k$-dimensional volume we need to require $k\leq N$, however this does not automatically imply that $k\leq M-N$. Nonetheless, the identity \eqref{detid} is still true regardless of the relationship of $k$ with $N$ or $M-N$. If there is a Parseval frame $\Phi\in\mathcal{P}(M,N)$ that has equal $k$-dimensional volumes for some $k>M-N$ then by Theorem \ref{equalvol} it also has equal $M-N$-dimensional volumes. If $\Psi\in\mathcal{P}(M,M-N)$ is a Naimark complement then Proposition \ref{naimarkvol} says that $\Psi$ has equal $M-N$-dimensional volumes and therefore has equal $k$-dimensional volumes for every $k\leq M-N$ again by Theorem \ref{equalvol}. Then if $|K|>M-N$ we have that $\det(\Psi_K^*\Psi_K)=0$ so by \eqref{detid} $\Phi$ has equal $k$-dimensional volumes for every $k\leq N$. 

While this proposition looks interesting we suspect that there are not very many examples of Parseval frames with equal volumes for $k>2$. If $\Phi$ is an equiangular Parseval frame then
$$
\Phi_{\{i,j\}}^*\Phi_{\{i,j\}}=
\begin{bmatrix}
\frac{N}{M} & \langle\varphi_i,\varphi_j\rangle \\
\langle\varphi_j,\varphi_i\rangle & \frac{N}{M}
\end{bmatrix}
$$
and since $|\langle\varphi_i,\varphi_j\rangle|=c_{M,N}$ not only is the determinant of this matrix the same for every set $\{i,j\}$, the individual eigenvalues are the same as well. It is known that over $\mathbb{R}$ it is not possible for every subset of size $k$ of a Parseval frame to have the same singular values if $k>2$ except in the trivial cases $M=N$ and $M=N+1$ (see \cite{BP05}), so while this does not rule out the possibility of having equal volumes it does suggest that if there were any such examples then the situation would be much more complicated. However in \cite{HS12} the authors identify a family of complex equiangular Parseval frames, which they refer to as 3-uniform, that do, in fact, have equal $3$-dimensional volumes (by Theorem \ref{equalvol} we can now say that these are the only Parseval frames with equal $3$-dimensional volumes), and they also suggest that there are very few nontrivial frames with equal $4$-dimensional volumes.

Even if there do exist some examples of Parseval frames with equal $k$-volumes we know that for most triples $(M,N,k)$ such frames will not exist since for most pairs $(M,N)$ equiangular Parseval frames do not exist. Therefore it is natural to ask whether Naimark complements to solutions to \eqref{volume} must also be solutions. This is not so clear, but we can say that they must be solutions to a related problem. Before presenting this statement we need to take a closer look at singular values of subsets of Parseval frames and their Naimark complements.

\begin{proposition}\label{svn}
Suppose $\Phi\in \mathcal{P}(M,N)$ and $\Psi\in \mathcal{P}(M,M-N)$ are Naimark complements and let $K\subseteq [M]$ with $|K|=k\leq N$. Then
$$
\sigma_i(\Psi_{K^c})=
\begin{cases}
1 \text{ if } 1\leq i\leq M-N-k\\
\sigma_j(\Phi_K) \text{ if } i=M-N-k+j 
\end{cases}
$$
for $i=1,...,M-N$. If $k>M-N$ then this formula is still true but in this case $\sigma_i(\Phi_K)=1$ for $i=1,...,k-(M-N)$.

\end{proposition}
\begin{proof}
Since $k\leq N$ the $N\times k$ matrix $\Phi_K$ can have $k$ nonzero singular values which are the square roots of the eigenvalues of the $k\times k$ matrix $\Phi_K^*\Phi_K$. Since $\Psi_{K}^*\Psi_{K}=I-\Phi_{K}^*\Phi_{K}$ it follows that $\sigma_i(\Psi_K)=\sqrt{1-\sigma_{k-i+1}^2(\Phi_{K})}$. Since $M-k\geq M-N$ the $(M-N)\times K$ matrix $\Psi_{K^c}$ can have $M-N$ nonzero singular values which are the square roots of the eigenvalues of $\Psi_{K^c}\Psi_{K^c}^*$. Since the $(M-N)\times(M-N)$ matrix $\Psi_K\Psi_K^*$ has the same nonzero eigenvalues as $\Psi_{K}^*\Psi_{K}$ the result now follows from the fact that $\Psi_K\Psi_K^*+\Psi_{K^c}\Psi_{K^c}^*=I$.

When $k>M-N$ the interlacing inequalities guarantee that $\Phi_K\Phi_K^*$ must have 1 as an eigenvalue with multiplicity $k-(M-N)$, see \cite{CFMPS13} and \cite{FMPS13}.
\end{proof}

Note that when $k>N$ then $M-k<M-N$ so we can apply Proposition \ref{svn} with the roles of $\Phi$ and $\Psi$ reversed. So what this proposition really says is that in all cases the singular values of $\Phi_K$ and the singular values of $\Psi_{K^c}$ are the same except for an appropriate number of 1's to fill in any empty spots. In particular, when $\Phi$ and $\Psi$ are Naimark complements,
$$
\prod\sigma_i(\Phi_K)=\prod\sigma_i(\Psi_{K^c})
$$
for every $K$. A special case of this is that if either of these products is nonzero then so is the other one, which means that the columns of $\Phi_K$ are linearly independent if and only if the columns of $\Psi_K$ span $\mathbb{F}^{M-N}$, see \cite{ACM12}.

For a collection of vectors $\{f_i\}_{i=1}^k\subseteq\mathbb{F}^M$ with $k>N$ we now define the \textit{complementary k-dimensional volume} as
$$
cv_k(\{f_i\}_{i=1}^k)=cv_k(F)=\sqrt{\det(FF^*)}=\prod_{i=1}^N\sigma_i(F).
$$
Note that the only difference between $cv_k(F)$ and $v_k(F)$ is the use of the frame operator $FF^*$ in place of the Gram matrix $F^*F$. We could allow $k>N$ in the definition of $v_k(F)$ (or $k<N$ in the definition of $cv_k(F)$) but then we would always have $v_k(F)=0$ (or $cv_k(F)=0$). In the case $k=N$ we can actually use either definition since $F$ is a square matrix so $\sqrt{\det(F^*F)}=\sqrt{\det(FF^*)}=|\det(F)|$. We will discuss this case in more detail shortly. For now we pose the following problem:

\begin{equation}\label{cvmax}
\max_{\Phi\in \mathcal{P}(M,N)}CV_k(\Phi):=\sum_{|K|=k}cv_k(\Phi_K).
\end{equation}

\begin{proposition}\label{vcv}
Let $k\leq N<M$ and suppose $\Phi\in \mathcal{P}(M,N)$ and $\Psi\in \mathcal{P}(M,M-N)$ are Naimark complements. Then $\Phi$ is a solution to \eqref{volume} for (M,N,k) if and only if $\Psi$ is a solution to \eqref{cvmax} for $(M,M-N,M-k)$.
\end{proposition}
\begin{proof}
By Proposition \ref{svn} $V_k(\Phi)=CV_{M-k}(\Psi)$.
\end{proof}

We will now briefly discuss the case $k=N$, for a more detailed account see section 4.3 of \cite{csag}. If $\Phi$ and $\Psi$ are Naimark complements the we have $V_N(\Phi)=CV_N(\Phi)=V_{M-N}(\Psi)=CV_{M-N}(\Psi)$, so \eqref{volume} and \eqref{cvmax} the same problem in this case. Recall that $\Phi$ is a Parseval frame if and only if the Gram matrix $\Phi^*\Phi$ is an orthogonal projection onto an $N$-dimensional subspace of $\mathbb{F}^M$, call this subspace $W_{\Phi}$ and observe that the rows of $\Phi$ form an orthonormal basis for $W_{\Phi}$. Conversely, given any subspace $W\in Gr(M,N)$ if we choose any orthonormal basis for $W$ and let $\Phi$ be the $N\times M$ matrix with these rows then the columns of $\Phi$ will be a Parseval frame in $\mathcal{P}(M,N)$.

Given any full rank $N\times M$ matrix $F$ consider the map $Plu(F)=(\det(F_K))_{|K|=N}$ and note that if $G$ is another $N\times M$ matrix with the same row space as $F$ then there is an invertible $N\times N$ matrix $A$ so that $G=AF$ and so $Plu(G)=\det(A)Plu(F)$. It is not too hard to show that the converse of this is also true, i.e., if $Plu(G)=\lambda Plu(F)$ then there is an invertible matrix $A$ with $\det(A)=\lambda$ such that $G=AF$ and so $F$ and $G$ have the same row space. Therefore $Plu$ assigns to each subspace $W\in Gr(M,N)$ a line in $\mathbb{F}^{{M\choose N}}$, or a point in $\mathbb{P}^{{M\choose N}-1}$. This is known as the Pl\"ucker embedding of the Grassmannian and the image of this map is a smooth, irreducible projective variety. The points in the image of the Pl\"ucker embedding can be characterized as satisfying a specific set of quadtratic polynomials known as the Pl\"ucker relations which we denote $Plu(M,N)$. A complete description of $Plu(M,N)$ is not needed here and many thorough references are readily available. What is important in our context is that we can now rewrite \eqref{volume} and equivalently \eqref{cvmax} as
\begin{align*}
\max_{x\in\mathbb{F}^{{M\choose N}}}&\|x\|_1 \\
\text{subject to } &\|x\|_2=1 \\
& Plu(M,N).
\end{align*}

In general $Plu(M,N)$ can contain many equations and this reformulation may not be very useful, however for small values of $M$ and $N$, $Plu(M,N)$ can be simpler than looking at the entries of $\Phi^*\Phi$. We will illustrate this in the case $\mathbb{F}=\mathbb{R}$, $N=2$, and $M=4$. Note that in this case there is no equiangular Parseval frame, so this is not covered by any of our previous results. Also, $Plu(4,2)$ only contains one equation. So our problem now becomes:
\begin{align*}
\text{max } &|x_{12}|+|x_{13}|+|x_{14}|+|x_{23}|+|x_{24}|+|x_{34}| \\
\text{subject to } &x_{12}^2+x_{13}^2+x_{14}^2+x_{23}^2+x_{24}^2+x_{34}^2 \\
& x_{12}x_{34}-x_{13}x_{24}+x_{14}x_{23}=0
\end{align*}
which can be solved by hand using Lagrange multipliers. This yields the optimal Pl\"ucker coordinates $(\frac{\sqrt{2}}{4},\frac{1}{2},\frac{\sqrt{2}}{4},\frac{\sqrt{2}}{4},\frac{1}{2},\frac{\sqrt{2}}{4})$ which correspond to the Parseval frame
$$
\Phi=
\left[\begin{array}{rcrc}
\frac{1}{2} & 0 & -\frac{1}{2} & -\frac{\sqrt{2}}{2} \\
\frac{1}{2} & \frac{\sqrt{2}}{2} & \frac{1}{2} & 0
\end{array}\right]
$$
which is two scaled orthonormal bases offset by 45 degrees. See \cite{csag} for the details of this calculation.

\section{Nuclear energy}\label{ne}

Given an $N\times k$ matrix $F$ the \textit{nucluear norm} of $F$ is
$$
\|F\|_*=\sum_{i=1}^{\min\{N,k\}}\sigma_i(F).
$$
We now pose the following problem for $k\leq M$:
\begin{equation}\label{nuke}
\max_{\Phi\in \mathcal{P}(M,N)}NE_k(\Phi):=\sum_{|K|=k}\|\Phi_K\|_*,
\end{equation}
where we refer to the quantity $NE_k(\Phi)$ as the \textit{k-nucluear energy} of $\Phi$.
For $\Phi\in \mathcal{P}(M,N)$ we have
\begin{eqnarray}
\sum_{|K|=k}\sum_{i=1}^{\min\{N,k\}}\sigma_i^2(\Phi_K)&=&\sum_{|K|=k}\text{trace}(\Phi_K\Phi_K^*) \notag \\
&=&\text{trace}(\sum_{|K|=k}\Phi_K\Phi_K^*)\notag \\
&=&\text{trace}(\sum_{|K|=k}\sum_{i\in K}\varphi_i\varphi_i^*)\notag \\
&=&\text{trace}({M-1\choose k-1}I)\notag \\
&=&N{M-1\choose k-1},\label{nukebound}
\end{eqnarray}
where the second to last line follows from the fact that each $\varphi_i\varphi_i^*$ appears exactly ${M-1\choose k-1}$ times. Therefore if $k\leq N$, $NE_k(\Phi)$ would be maximized if
$$
\sigma_i^2(\Phi_K)=N{M-1\choose k-1}{M\choose k}^{-1}k^{-1}=\frac{N}{M}
$$
for every $K$ and every $i$, and when $k\geq N$ if
$$
\sigma_i^2(\Phi_K)=N{M-1\choose k-1}{M\choose k}^{-1}N^{-1}=\frac{k}{M}
$$
for every $K$ and every $i$. However, this is not possible except in the trivial case $M=N$ where $\Phi$ is an orthonormal basis, or in the cases $k=1$ where this says $\text{trace}(\varphi_i\varphi_i^*)=\|\varphi_i\|^2=N/M$ for every $i$ (so $\Phi$ is an equal norm Parseval frame), or $k=M$ in which case $\sigma_i(\Phi\Phi^*)=1$ for every $i$ which is satisfied by every Parseval frame. In any other case this would mean that every subset of $\Phi$ of size $k$ was a tight frame for its span with bound $N/M$ which cannot happen, in particular if $k\leq N$ then there must be at least one linearly independent subset of size $k$ so this would mean that every subset of size $k$ was a scaled orthonormal set. Nonetheless, we can still try to minimize the variance amongst these singular values. To this end we define the \textit{k-nuclear variance} of a Parseval frame $\Phi$ as
$$
NVar_k(\Phi)=\sum_{|K|=k}\sum_{i=1}^{\min\{N,k\}}(\sigma_i(\Phi_K)-n_{\Phi,k})^2
$$ 
where
$$
n_{\Phi,k}=\min\{N,k\}^{-1}{M\choose k}^{-1}NE(\Phi)
$$
is the average of all of the singular values. By now it is a routine calculation to see that
$$
NVar_k(\Phi)=N{M-1 \choose k-1}-{M\choose k}^{-1}NE_k^2(\Phi)
$$
so that the $k$-nuclear variance is minimized precisely when the $k$-nuclear energy is maximized.

\begin{theorem}\label{eanuke}
The solutions to \eqref{nuke} for $k=2$ are precisely the equiangular Parseval frames when they exist.
\end{theorem}
\begin{proof}
Let $\Phi\in \mathcal{P}(M,N)$.
\begin{eqnarray*}
NE_2(\Phi)=\sum_{|K|=2}\|\Phi_K\|_*&\leq&\left(\frac{M(M-1)}{2}\sum_{|K|=2}\|\Phi_K\|_*^2\right)^{1/2} \\
&=&\left(\frac{M(M-1)}{2}\sum_{|K|=2}(\sigma_1(\Phi_K)+\sigma_2(\Phi_K))^2\right)^{1/2} \\
&=&\left(\frac{M(M-1)}{2}\sum_{|K|=2}\sigma_1^2(\Phi_K)+\sigma_2^2(\Phi_K)+2\sigma_1(\Phi_K)\sigma_2(\Phi_K)\right)^{1/2} \\
&=&(\frac{M(M-1)}{2}(N(M-1)+2V_2(\Phi))^{1/2} \\
&=&\sqrt{\frac{1}{2}MN(M-1)^2+M(M-1)V_2(\Phi)}
\end{eqnarray*}
with equality if and only if $\|\Phi_K\|_*=c$ for every $K$ and where we have used \eqref{nukebound}. Furthermore, this quantity will be maximized when $V_2(\Phi)$ is maximized.

If $\Phi$ is equiangular then there are constants $\sigma_1$ and $\sigma_2$ so that $\sigma_1(\Phi_K)=\sigma_1$ and $\sigma_2(\Phi_K)=\sigma_2$ for every $K$, therefore $\|\Phi_K\|_*=\sigma_1+\sigma_2$ is independent of $K$ which means $\Phi$ saturates the bound above. Now the result follows from Corollary \ref{eav} since equiangular Parseval frames maximize $V_2$.
\end{proof}

Before proceeding we make a few remarks about the above proof. First of all, this proof should be compared to the proof of Theorem \ref{main} and the derivation of the Welch bound, in that we needed to show that two inequalities were satisfied with equality. In this case we needed to show that equiangular Parseval frames have equal 2-nuclear norms in order to saturate the inequality in the proof and then we needed to know that they have equal 2-volumes in order to maximize the right hand side of the inequality. Because of this, we cannot use a similar argument to show that the solutions to \eqref{nuke} when $k=2$ are the same as the solutions to \eqref{volume} when $k=2$ (in fact, we suspect this is not the case) since we do not know if these solutions have equal 2-nuclear norms and we know that they do not have equal 2-volumes.

Next we remark that even if we were able to find examples of Parseval frames with equal $k$-volumes for $k>2$ we would not be able to use a similar method to show that these must be solutions to \eqref{nuke} because when we expand $\|\Phi_K\|_*^2$ we will have products of pairs of singular values of $\Phi_K$ but these do not represent the 2-volumes of $\Phi$ since $k>2$. Considering $\|\Phi_K\|_*^p$ for $p\neq 2$ would not fix this problem. In fact, even if we could find examples of Parseval frames that have equal $k$-nuclear norms we cannot immediately guarantee that they would be solutions to \eqref{nuke}. In this regard it is tempting to consider the problem
\begin{equation}\label{nukevar}
\min_{\Phi\in \mathcal{P}(N,M)}\sum_{|K|=k}(\|\Phi_K\|_*-d_{\Phi,k})^2
\end{equation}
where $d_{\Phi,k}={M\choose k}^{-1}NE_k(\Phi)$ is the average of the $k$-nuclear norms. While this is a perfectly reasonable problem (and it would clearly be minimized if $\|\Phi_K\|_*=d_{\Phi,k}$ for every $K$) we cannot immediately guarantee that solutions to this problem are solutions to \eqref{nuke} since $\sum\|\Phi_K\|_*^2$ is not determined by $(M,N,k)$.

We conclude this section with a result similar to Proposition \ref{vcv} for the vase of \eqref{nuke}.

\begin{proposition}
If $\Phi\in\mathcal{P}(M,N)$ and $\Psi\in\mathcal{P}(M,M-N)$ are Naimark complements then $\Phi$ is a solution to \eqref{nuke} for $(M,N,k)$ if and only if $\Psi$ is a solution to \eqref{nuke} for $(M,M-N,M-k)$.
\end{proposition}
\begin{proof}
First suppose $k\leq N$. If $k\leq M-N$ then by Proposition \ref{svn}
$$
\|\Psi_{K^c}\|_*=\|\Phi_K\|_*+M-N-k
$$
for every $K\subseteq [M]$ with $|K|=k$. Therefore
$$
NE_{M-k}(\Psi)=\sum_{|K|=k}\|\Psi_{K^c}\|_*=NE_k(\Phi)+(M-N-k){M\choose k}.
$$

Similarly, if $k>M-N$ then by Proposition \ref{svn}
$$
\|\Phi_K\|_*=\|\Psi_{K^c}\|_*+k-(M-N)
$$
for every $K\subseteq[M]$ with $|K|=k$, so
$$
NE_k(\Phi)=NE_{M-k}(\Psi)+(k-(M-N)){M\choose k}.
$$

In either case $NE_k(\Phi)$ and $NE_{M-k}(\Psi)$ differ by a constant that depends only on $M,N$, and $k$ but not on $\Phi$ or $\Psi$.

Finally, if $k>N$ then $M-k<M-N$ so we can apply the same argument with the roles of $\Phi$ and $\Psi$ reversed.
\end{proof}

\section{Equal norm frames}\label{ens}

In this section we will discuss some results for equal norm (but not necessarily Parseval) frames. Usually this kind of dicussion would focus on unit norm frames (i.e., frames with $\|\varphi_i\|=1$ for every $i$) but in order to be consistent with the rest of this paper we will consider frames $\{\varphi_i\}_{i=1}^M$ for $\mathbb{F}^N$ that satisfy $\|\varphi_i\|=\sqrt{N/M}$ for every $i$. To this end let
$$
\mathcal{E}(M,N)=\{\{\varphi_i\}_{i=1}^M\subseteq\mathbb{F}^N:\|\varphi_i\|=\sqrt{\frac{N}{M}}\text{ for every }i\in [M]\}.
$$
Note that we do not require that the elements of $\mathcal{E}(M,N)$ are frames. Note also that if $\Phi\in\mathcal{E}(M,N)$ is a tight frame then it is a Parseval frame by our choice of normalization.

First consider problem \eqref{2} where the optimization is over $\mathcal{E}(M,N)$ rather than $\mathcal{P}(M,N)$. In this case it is not hard to see that $TC(\Phi)$ will be maximized when $\varphi_1=\cdots=\varphi_M$ which is not a frame. Therefore this problem is not interesting as stated. A slightly more interesting problem is to try to minimize $TC(\Phi)$ over $\mathcal{E}(M,N)$ in which case the solutions are repeated copies of a scaled orthonormal basis, see \cite{EO12}.

We now turn our attention to problem \eqref{volume} for the case of equal norm frames. To be precise, we consider the problem
\begin{equation}\label{envol}
\max_{\Phi\in\mathcal{E}(M,N)}V_k(\Phi).
\end{equation}

\begin{theorem}\label{fp}
If $\Phi\in\mathcal{E}(M,N)$ and $k\leq N$ then
$$
\sum_{|K|=k}v_k^2(\Phi_K)\leq {N \choose k}
$$
with equality if and only if $\Phi$ is Parseval.
\end{theorem}
\begin{proof}
Let $\lambda_1,...,\lambda_N$ be the eigenvalues of the frame operator $\Phi\Phi^*$ which are the same as the nonzero eigenvalues of the Gram matrix $\Phi^*\Phi$. Let $p(\lambda)=\prod_{i=1}^N(\lambda-\lambda_i)=\lambda^N+a_{N-1}\lambda^{N-1}+\cdots+a_1\lambda+a_0$ be the characteristic polynomial of $\Phi\Phi^*$ and note that
$$
a_{N-k}=(-1)^k\sum_{|K|=k}\prod_{i\in K}\lambda_i.
$$

Now observe that $\sum_{|K|=k}v_k^2(\Phi_K)$ is the sum of all principal $k\times k$ minors of the Gram matrix $\Phi^*\Phi$ and therefore $(-1)^k\sum_{|K|=k}v_k^2(\Phi_K)$ is the coefficient of $\lambda^{M-k}$ in the characteristic polynomial of $\Phi^*\Phi$ (see section 7.1 in \cite{LAbook}), but the characteristic polynomial of $\Phi^*\Phi$ is $\lambda^{M-N}p(\lambda)$, so we have
\begin{equation}\label{w1}
\sum_{|K|=k}v_k^2(\Phi_K)=\sum_{|K|=k}\prod_{i\in K}\lambda_i.
\end{equation}
Maclaurin's inequality says that
\begin{equation}\label{w2}
\sum_{|K|=k}\prod_{i\in K}\lambda_i\leq (\frac{1}{N}\sum_{i=1}^N\lambda_i)^k{N\choose k},
\end{equation}
with equality if and only if $\lambda_1=\cdots=\lambda_N$, which means $\Phi$ is Parseval. Finally,
\begin{eqnarray}
\frac{1}{N}\sum_{i=1}^N\lambda_i&=&\text{trace}(\Phi\Phi^*) \notag  \\ 
&=&\text{trace}(\Phi^*\Phi) \notag \\
&=&\frac{1}{N}\sum_{i=1}^M\|\varphi_i\|^2=(\frac{1}{N})(M)(\frac{N}{M})=1. \label{w3}
\end{eqnarray}
The result now follows by combining \eqref{w1}, \eqref{w2}, and \eqref{w3}.
\end{proof}

Aside from giving us a bound on $\sum v_k^2(\Phi_K)$, Theorem \ref{fp} also tells us that this quantity is maximized (over $\mathcal{E}(M,N)$) precisely by the Parseval frames. For the special case $k=2$ we have $v_2^2(\Phi_{\{i,j\}})=\frac{N^2}{M^2}-|\langle\varphi_i,\varphi_j\rangle|^2$, so $\sum v_2^2(\Phi_{\{i,j\}})$ is maximized precisely when $\sum_{i\neq j}|\langle\varphi_i,\varphi_j\rangle|^2$ is minimized. This latter quantity is known as the \textit{frame potential} and the well known result in \cite{BF03} states that this is minimized by tight frames. Therefore Theorem \ref{fp} can be seen as a generalization of this result to the cases $k>2$.

\begin{corollary}
If $\Phi\in\mathcal{E}(M,N)$ and $k\leq N$ then
\begin{equation}\label{tvwelch}
V_k(\Phi)\leq \sqrt{{M\choose k}{N \choose k}}
\end{equation}
and
\begin{equation}\label{volwelch}
\min_{|K|=k}v_k(\Phi_K)\leq c_{M,N,k}
\end{equation}
with equality if and only if $\Phi$ is Parseval and $v_k(\Phi_K)=c_{M,N,k}$ for every $K$ with $|K|=k$.
\end{corollary}

If we now consider \eqref{volwelch} for the special case $k=2$ we see that
$$
\min_{\{i,j\}}v_2^2(\Phi_{\{i,j\}})\leq\frac{N(N-1)}{M(M-1)}.
$$
Again using the observation that $v_2^2(\Phi_{\{i,j\}})=\frac{N^2}{M^2}-|\langle\varphi_i,\varphi_j\rangle|^2$ we see that this is equivalent to
\begin{eqnarray*}
\max_{i\neq j}|\langle\varphi_i,\varphi_j\rangle|^2&\geq&\frac{N^2}{M^2}-\frac{N(N-1)}{M(M-1)} \\
&=&\frac{N(M-N)}{M^2(M-1)}=c_{M,N}^2.
\end{eqnarray*}

Recall that we have chosen to scale our vectors so that $\|\varphi_i\|=\sqrt{N/M}$. If we rescale so that $\|\varphi_i\|=1$ the above inequality becomes
$$
\max_{i\neq j}|\langle\varphi_i,\varphi_j\rangle|\geq\sqrt{\frac{M-N}{N(M-1)}}
$$
which is the well known Welch bound \cite{W74}. Therefore \eqref{volwelch} can be seen as a generalization of the Welch bound to the cases $k>2$. It is worth noting that as per the discussion in Section \ref{tv}, this bound cannot be saturated very often.

The problem
\begin{equation}\label{gf}
\min_{\Phi\in\mathcal{E}(M,N)}\max_{i\neq j}|\langle\varphi_i,\varphi_j\rangle|
\end{equation}
was initially posed in \cite{SH03} and is very well studied. The solutions to this problem are called \textit{Grassmannian frames} because they correspond to optimal line packings, i.e., optimal point configurations in the Grassmannian of $1$-dimensional subspaces. See \cite{CHS96} for a good introduction and \cite{JKM19} for an up to date survey of these ideas.

We now remark that even though \eqref{envol} is a well defined problem and the solutions coincide with the solutions to \eqref{volume} at least in some cases, we cannot interpret these solutions as minimizing the variance as in \eqref{volvar}. This is because $\sum v_k^2(\Phi)$ is not fixed (similar to the situation with \eqref{nukevar}) but only bounded. Indeed, if we take an equiangular Parseval frame and remove one (or several) vectors then the remaining set is still equiangular (but not Parseval) and so the variance in the $2$-dimensional volumes would be 0, but such a frame need not be a solution to \eqref{envol}.

Even though we do not have a notion of Naimark complement for equal norm frames the definition of $CV_k(\Phi)$ still makes sense, so we can consider the analog of \eqref{cvmax} where we optimize over $\mathcal{E}(M,N)$ rather than $\mathcal{P}(M,N)$, i.e.,
\begin{equation}\label{encv}
\max_{\Phi\in\mathcal{E}(M,N)}CV_k(\Phi).
\end{equation}

\begin{proposition}
If $\Phi\in\mathcal{E}(M,N)$ and $k\geq N$ then
$$
CV_k(\Phi)\leq\sqrt{{M\choose k}{M-N\choose M-k}}.
$$
with equality if and only if $\Phi$ is Parseval and
$$
cv_k(\Phi_K)=\sqrt{{M\choose k}^{-1}{M-N\choose M-k}}
$$
for every $K\subseteq[M]$ with $|K|=k$.
\end{proposition}
\begin{proof}
This is essentially the same as the proof of Theorem 1 in \cite{MB96} but we will include it for completeness.

For a fixed subset $K\subseteq[M]$ with $|K|=k$ we have
\begin{eqnarray*}
cv_k^2(\Phi_K)&=&\det(\Phi_K\Phi_K^*)=\sum_{\substack{J\subseteq K \\ |J|=N}}\det(\Phi_J\Phi_J^*)
\end{eqnarray*}
by the Cauchy-Binet formula. Therefore
\begin{eqnarray*}
\sum_{|K|=k}cv_k^2(\Phi_K)&=&\sum_{|K|=k}\sum_{\substack{J\subseteq K \\ |J|=N}}\det(\Phi_J\Phi_J^*) \\
&=&{M-N\choose M-k}\sum_{|J|=N}\det(\Phi_J\Phi_J^*) \\
&=&{M-N\choose M-k}\det(\Phi\Phi^*)
\end{eqnarray*} 
where the last line again follows from the Cauchy-Binet formula. Since $\Phi\in\mathcal{E}(M,N)$ the arithmetic-geometric mean inequality says that $\det(\Phi\Phi^*)\leq 1$ with equality if and only if $\Phi$ is Parseval. Therefore we have shown that
$$
\sum_{|K|=k}cv_k^2(\Phi_K)\leq{M-N\choose M-k}
$$
from which the result readily follows.
\end{proof}

We now turn our attention to the $k$-nuclear energy $NE_k(\Phi)$ for $\Phi\in\mathcal{E}(M,N)$,i.e., we consider the problem
\begin{equation}\label{enne}
\max_{\Phi\in\mathcal{E}(M,N)}NE_k(\Phi).
\end{equation}
In this case we have
\begin{eqnarray*}
\sum_{|K|=k}\sum_{i=1}^{\min\{k,N\}}\sigma_i^2(\Phi_K)&=&\sum_{|K|=k}\text{trace}(\Phi_K^*\Phi_K) \\
&=&\sum_{|K|=k}\sum_{i\in K}\|\varphi_i\|^2 \\
&=&\frac{kN}{M}{M\choose k}=N{M-1\choose k-1}.
\end{eqnarray*}

We can use this to derive an upper bound on $NE_k(\Phi)$. However, as is the case for Parseval frames, this bound would only be saturated when all of the singular values of every subset of size $k$ were equal, which cannot happen except in the trivial cases.

The same argument used to prove Theorem \ref{eanuke} combined with \eqref{tvwelch} proves the following:

\begin{proposition}
The solutions to
$$
\max_{\Phi\in\mathcal{E}(M,N)}NE_2(\Phi)
$$
are precisely the equiangular Parseval frames when they exist.
\end{proposition}

\section{Discussion and conclusion}\label{dc}

In this paper we have stated several optimization problems for which equiangular Parseval frames are the solutions when they exist. However, all of the functions we consider are continuous, and since $\mathcal{P}(M,N)$ and $\mathcal{E}(M,N)$ are compact, the corresponding optimization problems always have solutions, even when there are no equiangular Parseval frames. Therefore the most natural question to ask is what are the solutions in the cases when there are no equiangular Parseval frames? However, we do not expect to ever be able to completely answer this question, primarily because the optimization problems we consider are hard. Nonetheless, there are many questions we can ask about solutions (or approximate solutions) to these problems that may have nice answers. Of course, the most obvious question is to find explicit examples that are solutions to any of the problems presented in this paper that are not equiangular Parseval frames, even for small values of $N,M$, and $k$.

The first set of questions we will pose relate to any structural properties that the solutions to these problems might have. One very natural question in this regard is: Do the solutions to \eqref{2}, \eqref{volume} and/or \eqref{nuke} have to be equal norm? By Theorem \ref{tcnorm} we know that there has to be some control on the norms of solutions to \eqref{2}, and we should be able to use a similar technique to show the same thing for the solutions to \eqref{volume} and \eqref{nuke}. Even if it turns out that the solutions are not equal norm in general it would be interesting to have quantitative bounds on how big and small the norms can be. Similarly, we can ask if the solutions to \eqref{envol}, \eqref{encv}, and \eqref{enne} need to be Parseval. It is known that solutions to \eqref{gf} need not be Parseval, see \cite{BK06} or \cite{CHS96}.

We could insist that our solutions be equal norm Parseval frames by maximizing the functions $TV(\Phi)$, $V_k(\Phi)$, or $NE_k(\Phi)$ over $\mathcal{P}(M,N)\cap\mathcal{E}(M,N)$. This set is compact and all of these functions are continuous so solutions certainly exist. This set is also known to be connected \cite{CMS17} and explicit formulas for the tangent spaces at the nonsingular points are known \cite{S11} so this type of optimization is possible, but the geometry of $\mathcal{P}(M,N)\cap\mathcal{E}(M,N)$ is much more complicated than that of either $\mathcal{P}(M,N)$ or $\mathcal{E}(M,N)$.

If the solutions to any of these problems do turn out to be equal norm Parseval frames they very likely have some additional structure. For example, biangular Parseval frames are Parseval frames where the off diagonal entries of the Gram matrix can have two different values in magnitude. While these do not exist for every $M$ and $N$ they do exist in many cases where equiangular Parseval frames do not. More generally, a frame is \textit{equidistributed} if the columns of the magnitudes of the Gram matrix are permutations of one another, these exist for every $M$ and $N$. In \cite{BH15} the authors show that under certain conditions the solutions to problems similar to \eqref{2} are equidistributed. We could say the $k$-dimensional volumes of a $\Phi$ are equidistributed if $\{v_k(\Phi_K):i\in K\}$ is the same for every $i\in[M]$. It is very likely that at least in some cases the solutions to \eqref{volume} and \eqref{envol} will exhibit some structure similar to this.

Just as important as properties that solutions to these problems could have are properties that they do not have. For examle, in Theorem \ref{tcnorm} we showed that the solutions to \eqref{2} do not contain any zero vectors, and a similar argument should work to show that the solutions to the other problems presented in this paper do not contain any zero vectors. However the way that we showed that the solutions to \eqref{2} do not contain zero vectors was to start with a frame that had a zero and then replace the zero with a scaled copy of one of the other vectors. But we suspect that the solutions to \eqref{2} as well as the other problems cannot contain any parallel vectors. A frame is called \textit{orthodecomposable} is there is a subset $S\subseteq[M]$ such that $\langle\varphi_i,\varphi_j\rangle=0$ whenever $i\in S$ and $j\not\in S$. We suspect that the solutions to all of the problems presented in this paper cannot be orthodecomposable.

In fact, from the perspective of \eqref{var} it looks like the magnitudes of the off diagonal entries of the Gram matrix of a solution to \eqref{var} (and equivalently \eqref{2}) are clustered very tightly around their mean which suggests, in particular, that $|\langle\varphi_i,\varphi_j\rangle|>0$, i.e., if $\Phi$ solves \eqref{2} then $\Phi$ does not contain any orthogonal pair of vectors. Similarly it is natural to think that if $\Phi$ is a solution to \eqref{volume} or \eqref{envol} then all $k$-dimensional volumes of $\Phi$ must be positive, but this is not immediately clear. We could ask a similar question about the solutions to \eqref{nuke}. We define $\text{spark}(\Phi)$ to be the size of the smallest linearly dependent subset of $\Phi$.  To pose a precise question: If $\Phi$ is a solution to \eqref{volume}, \eqref{nuke}, \eqref{envol}, or \eqref{enne} is it necessarily true that $\text{spark}(\Phi)>k$?

On a related note it is natural to ask if there is any relationship between the solutions to \eqref{volume} for different values of $k$. For example, by Theorem \ref{equalvol} we know that if $\Phi$ has equal $k$-dimensional volumes then it is a solution to \eqref{volume} for every $j\leq k$, however we also know that there are not very many examples of Parseval frames with equal $k$-dimensional volumes. Is it possible for a given frame $\Phi$ to be solutions to \eqref{volume} for multiple values of $k$? If $\Phi$ is a solution to \eqref{volume} for a certain value of $k$ does it have to be a solution for $k-1$ as well? One thing worth noting in thus regard is that many equiangular Parseval frames have a relatively low spark, see \cite{FMT12}, so if it is true that solutions to \eqref{volume} must have $\text{spark}(\Phi)>k$ then this would rule these frames out as solutions for many values of $k$ even though they are solutions for $k=2$. We can ask the same questions about the solutions to \eqref{nuke}, \eqref{envol}, \eqref{encv}, or \eqref{enne} as well.

It is also natural to wonder whether there is any relationship between the solutions to \eqref{volume}, \eqref{cvmax}, and \eqref{nuke} for a fixed value of $k$ (and similarly for \eqref{envol}, \eqref{encv}, and \eqref{enne}). Is it possible for the same frame to simultaneously be solutions to \eqref{volume} and \eqref{nuke} for $k>2$? Are Naimark complements of solutions to \eqref{volume} also solutions to \eqref{cvmax} in general (for the same $k$)? If they are then Proposition \ref{vcv} would say that the solutions to \eqref{volume} and \eqref{cvmax} are the same. Similarly we can ask if the solutions to \eqref{envol} and \eqref{encv} are the same. We can also ask the same question about the nuclear energy, i.e., are the solutions to \eqref{nukevar} the same for $k$ and $M-k$ (and the same for \eqref{enne}). Another question is if the solutions to \eqref{nukevar} are, in fact, solutions to \eqref{nuke}.

Another type of question that we can consider is that of getting better bounds on the optimal values of the functions that we are optimizing. For example, in Proposition \ref{tcbound} we showed that $TC(\Phi)\geq\max\{N,M-N\}$ for any equal norm frame, but if $\Phi$ is a solution to \eqref{2} then $TC(\Phi)$ is most likely quite a bit bigger than that, so it is desirable to have a tighter lower bound. For example, one natural question is does there exist a constant $c>0$ independent of $N$ and $M$ such that there always exists  Parseval frame $\Phi$ with
$$
TC(\Phi)\geq c\sqrt{N(M-N)(M-1)}?
$$
We can ask a similar question about $V_k(\Phi)$ with the upper bound given in Proposition \ref{volbound} (and \eqref{volwelch}, which is the same bound). Of course it is also desirable to have tighter upper bounds in the cases where the bounds we have cannot be saturated. For the case of $NE_k(\Phi)$ we don't even have a good upper bound since the bound we get from \eqref{nukebound} can never be saturated so a bound in either direction is desirable.

One way to approach finding bounds on the optimal values of these functions is to try to calculate the expected value of them for a random Parseval frame or a random equal norm frame. It is clear what is meant by a random equal norm frame. For a random Parseval frame recall the correspondence between unitary equivalence classes of Parseval frames in $\mathcal{P}(M,N)$ and the Grassmannian of $N$-dimensional subspaces of $\mathbb{F}^M$. The uniform distribution on the Grassmannian is a well understood distribution so when we say random Parseval frame we really mean the unitary equivalence class of Parseval frames whose Gram matrix is the projection onto a subspace chosen from this distribution. Since all of the functions we have considered are invariant under unitary transformations this is a viable method to obtain the desired bounds. Aside from providing bounds on the optimal values, it is quite possible that randomly chosen Parseval frames or equal norm frames will be very close to optimal which would say that there is some kind of concentration of measure.

Many applications of frame theory call for frames that have some property that requires every subset of a certain size to have some nice property. Often times researchers try to understand these properties in terms of the \textit{coherence} of the frame, that is $\max_{i\neq j}|\langle\varphi_i,\varphi_j\rangle|$. While the coherence does a good job of capturing the behavior of subsets of size 2, it generally does not do a very good job of capturing the behavior of larger subsets. Therefore, a notion of coherence for larger subsets is needed. The quantities $V_k(\Phi)$ and $NE_k(\Phi)$ are two such candidates. Probably the most famous such example is compressed sensing, which is the problem of recovering sparse vectors from incomplete linear measurements via $\ell_1$-minimization. A thorough introduction to compressed sensing is beyond the scope of this paragraph, but we refer to the books \cite{EK12,FR13}. One of the most important properties in the theory of compressed sensing is the \textit{restricted isometry property} introduced in \cite{CT05} which says that for an $N\times M$ matrix $\Phi$ there is a $\delta>0$ so that
$$
(1-\delta)\|x\|^2\leq\|\Phi x\|^2\leq(1+\delta)\|x\|^2
$$
whenever $x$ is $k$-sparse (i.e., $x$ has at most $k$ nonzero entries). It is clear from the definition that an RIP-matrix should have relatively high $k$-nuclear energy, so the question is if the converse of this is true, i.e., if $\Phi$ has a high enough $k$-nuclear energy does it need to be an RIP-matrix? More specifically, are the solutions to \eqref{nuke} and \eqref{enne} RIP-matrices? Are they, in some sense, the best possible RIP-matrices? If they are not then are they good for compressed sensing in some other sense? Also, it seems intuitively true that the solutions to \eqref{volume} and \eqref{envol} should also be good for compressed sensing, but it is not clear from the definition that they need to have very good RIP properties, so a more open ended question is to what extent can matrices that have high total $k$-dimensional volume be used in compressed sensing?

Another problem of a similar flavor is that of robustness to erasures. This line of work was initiated in \cite{GKK01} (see also \cite{CK03,HP04,PC05} for some earlier work on this topic), loosely speaking this problems asks for frames that have the property that after removing a certain number of vectors the remaining set is still a pretty good frame. More specifically, in \cite{FM12} the authors define a $(k,C)$-\textit{numerically erasure robust frame} (NERF) to be a frame that has the property that after removing any set of $k$ vectors the remaining set has condition number at least $C$. For a Parseval frame we know that the singular values of any subset are bounded by 1, and for large subsets they are equal to 1 or very close to 1, so to be a good NERF we would want all of the singular values of every subset of size $M-k$ to be as big as possible. Thus it seems that the solutions to \eqref{nuke} (for $M-k$) should be very good NERFs. It also seems like the solutions to \eqref{cvmax} should have good erasure robustness properties as well, but in exactly what sense is not clear.

\section{Acknowledgements}
The second author was supported by NSF DMS 1609760 and 1906725.

\bibliographystyle{plain}
\bibliography{l1max}

\end{document}